\documentclass[11pt, reqno]{amsart}
\usepackage{amsmath,amssymb,amsbsy,amsfonts,amsthm,latexsym,amsopn,amstext,amsxtra,epic, euscript,amscd,indentfirst,verbatim,graphicx, hyperref,xcolor, setspace,tikz,tikz-cd}
\usepackage[normalem]{ulem} 
\usepackage{tabularx} 
\usepackage{relsize} 
\usepackage[margin=1.2in]{geometry}
\usepackage{enumitem} 
\usepackage{hyperref} 
\newcommand{\stkout}[1]{\ifmmode\text{\sout{\ensuremath{#1}}}\else\sout{#1}\fi} 

\DeclareMathOperator{\Tr}{Tr}
\DeclareMathOperator{\Ric}{Ric}

\DeclareMathOperator{\Hess}{Hess}

\begin{document}

\numberwithin{equation}{section}
\newtheorem{theorem}{Theorem}[section]
\newtheorem*{theorem*}{Theorem}
\newtheorem{theoremstar}[theorem]{Theorem*}
\newtheorem{conjecture}[theorem]{Conjecture}
\newtheorem*{conjecture*}{Conjecture}
\newtheorem{proposition}[theorem]{Proposition}
\newtheorem*{proposition*}{Proposition}
\newtheorem*{``proposition"*}{``Proposition"}
\newtheorem{question}{Question}
\newtheorem{lemma}[theorem]{Lemma}
\newtheorem*{lemma*}{Lemma}
\newtheorem{cor}[theorem]{Corollary}
\newtheorem*{obs*}{Observation}
\newtheorem{obs}{Observation}
\newtheorem{example}[theorem]{Example}
\newtheorem{condition}{Condition}
\newtheorem{definition}[theorem]{Definition}
\newtheorem*{definition*}{Definition}
\newtheorem{proc}[theorem]{Procedure}
\newtheorem{problem}{Problem}
\newtheorem{remark}{Remark}
\newcommand{\comments}[1]{} 
\def\Z{\mathbb Z}
\def\Za{\mathbb Z^\ast}
\def\Fq{{\mathbb F}_q}
\def\R{\mathbb R}
\def\N{\mathbb N}
\def\C{\mathbb C}
\def\k{\kappa}
\def\grad{\nabla}
\def\M{\mathcal{M}}
\def\S{\mathcal{S}}
\def\pt{\partial}

\newcommand{\todo}[1]{\textbf{\textcolor{red}{[To Do: #1]}}}
\newcommand{\note}[1]{\textbf{\textcolor{blue}{#1}} \\ \\}

\title[Fundamental Gaps under Conformal Change]{Concavity Properties of Solutions of Elliptic Equations under Conformal Deformations}
 \author[Iowa State University]{Gabriel Khan} 
 \address[Gabriel Khan]{Department of Mathematics, Iowa State University, Ames, IA, USA.}
  \email{gkhan@iastate.edu}  
 \thanks{G. Khan is supported in part by Simons Collaboration Grant 849022}

  \author[]{Soumyajit Saha} 
\address[Soumyajit Saha]{Department of Mathematics, Iowa State University, Ames, IA, USA.}  \email{ssaha1@iastate.edu}

   \author[UCSB]{Malik Tuerkoen}
   \address[Malik Tuerkoen]{Department of Mathematics, University of California,  Santa Barbara, CA, USA.}
  \email{mmtuerkoen@ucsb.edu}

\date{\today}

\singlespace

\maketitle 

\allowdisplaybreaks

\begin{abstract}

 We study the Dirichlet problem for the weighted Schr\"odinger operator \[-\Delta u +Vu  = \lambda \rho u,\] where $\rho$ is a positive weighting function and $V$ is a potential. Such equations appear naturally in conformal geometry and in the composite membrane problem. Our primary goal is to establish concavity estimates for the principle eigenfunction with respect to conformal connections. Doing so, we obtain new bounds on the fundamental gap problem, which is the difference between the first and second eigenvalues. In particular, we partially resolve a conjecture of Nguyen, Stancu and Wei \cite{nguyen2022fundamental} on the fundamental gap of horoconvex domains. In addition, we obtain a power convexity estimate for solutions to the torsion problem in spherical geometry on convex domains which are not too large.

\end{abstract}


\section{Introduction}

In this article, we use conformal deformations to show that the ground-state eigenfunction $u_1$ of a convex domain $\Omega$ satisfies certain log-convexity properties with respect to a suitably chosen connection. More precisely, given a bounded domain $\Omega$ in a Riemannian manifold $(M,g)$, we consider the Dirichlet eigenvalue problem
\begin{eqnarray}\label{Dirichlet-Laplacian}
    -\Delta _g u = \lambda u \quad\textup{in }\Omega, \quad \quad u = 0 \quad\textup{on }\partial \Omega,
\end{eqnarray}
where $\Delta_g$ is the Laplace-Beltrami operator.
There is a discrete sequences of eigenvalues 
$$0<\lambda_1 < \lambda_2\leq \cdots \leq \lambda_k\leq \cdots \nearrow \infty,$$
repeated with multiplicity as well as associated eigenfunctions $u_i$ which form an orthogonal basis of $L^2(\Omega)$.
 By changing the underlying metric in a conformal fashion, the eigenvalue equation changes as well. In particular, if $\tilde g = e^{2\varphi}g$ and $\tilde u$ is a $\tilde g$ Dirichlet eigenfunction, then $u = e^{\tfrac{n-2}{2}\varphi}\tilde  u$ solves a Schr\"odinger equation of the form  \begin{eqnarray}\label{eqn: Weighted eigenfunction equation}
 -\Delta_g u + V u = \lambda \rho u \quad\textup{in }\Omega, \qquad u = 0 \quad\textup{on }\partial \Omega,
\end{eqnarray}
where $V$ is a potential and $\rho$ is a weighting function (see Section \ref{Background Section} for details). Equation \ref{eqn: Weighted eigenfunction equation} is also known as a \emph{composite membrane problem}, because the weighting function $\rho$ describes objects whose density varies throughout the domain \cite{cox1990extremal,chanillo2000symmetry}.

One central motivation for establishing log-concavity estimates for the principle eigenfunction is that doing so allows us to control \emph{fundamental gap}, which is the difference between the lowest two eigenvalues of Equation \eqref{eqn: Weighted eigenfunction equation}  
\begin{align*}
    \Gamma(\Omega)=\lambda_2-\lambda_1>0.
\end{align*}
In independent works of van den Berg, Ashbaugh-Benguria, and Yau \cite{ashbaugh, vandenBerg, yauconjecture}, it was conjectured that for any convex domain $\Omega \subset \mathbb{R}^n$ and a convex potential $V$, the fundamental gap is at least $\frac{3\pi^2}{D^2}$, where $D$ is the diameter of the domain. This conjecture was proven by Andrews and Clutterbuck in 2011 \cite{andrews2011proof} by proving a strong log-concavity estimate for the first eigenfunction (see also, \cite{brascamp1976extensions,singer1985estimate,yu1986lower,Gong-Li-Luo2016} for related results). There has been a large amount of work on estimating the fundamental gap in spherical geometry \cite{lee1987estimate, wang2000estimation, 10.4310/jdg/1559786428,cho2023probabilistic} and for deformations of round spheres in dimension two \cite{surfacepaper1,khan2023modulus}, where the main challenge was to prove log-concavity estimates of the first eigenfunction.

For spaces of negative curvature, the fundamental gap can be much smaller (see Section \ref{Fundamental Gap in hyperbolic space subsection} for a discussion on this). In particular, in hyperbolic space it is possible to find convex domains of any diameter $D,$ whose fundamental gap is arbitrarily small \cite{bourni2022vanishing} (see also \cite{khan2022negative}). 
Since assuming only convexity is not enough to ensure uniform lower bounds on the fundamental gap in such geometries, it is natural to consider stronger convexity assumptions. In particular, Nguyen, Stancu and Wei  \cite{nguyen2022fundamental} posed the question of whether it is possible to find lower bounds of the fundamental gap of a \emph{horoconvex} domain in terms of its diameter. In this article, we provide a partial answer to this question.
\begin{theorem}
    \label{Fundamental gap of Horoconvex-domains}
    Let $\Omega \subset \mathbb{H}^2$ be a horoconvex domain whose diameter satisfies
    \[D<2 \, \textup{arccsch}(2 \sqrt{11/3} )\approx 0.516474.\] Then the fundamental gap of $\Omega$ satisfies \begin{align*}
        \Gamma(\Omega) \geq \frac{32}{3 (7 + \sqrt{33})} \frac{\pi^2}{ D^2} + \frac{4}{3} \approx  0.837 \frac{\pi^2}{ D^2} + \frac{4}{3}.
    \end{align*}
\end{theorem}
 The key to proving this estimate is to establish the log-concavity of the first eigenfunction with respect to a connection that has constant \emph{positive} curvature. In prior work, the log-concavity of eigenfunctions has been studied solely in terms of the Levi-Civita connection of the associated metric. However, by considering a conformal deformation, this induces a choice of a connection on the tangent bundle and one can study concavity properties of the solution to \eqref{eqn: Weighted eigenfunction equation} under the connection induced by $g$ in contrast to the one induced by $\tilde g.$
\begin{obs}\label{Log-concavity and fundamental gaps}
    It is possible to establish concavity estimates for solutions to elliptic PDEs by choosing the connection carefully.
\end{obs}

The idea of choosing a connection to prove analytic results has a long history in differential geometry (see, e.g., \cite{uhlenbeck1982connections,uhlenbeck1982removable}) and is often known as \emph{gauge theory}. 
From this perspective, choosing a conformal deformation and considering the associated connection corresponds to a conformal gauge. Traditionally, conformal gauge theory has been developed using the theory of principle bundles (see, for instance, \cite{ogiue1967theory}). 
In this paper, we will not use this approach and so our choices of gauge will appear as particular choices of coordinates and conformal factors. However, it is possible to express our results in coordinate-free language by reformulating them in gauge-theoretic terms.

\subsection{Overview of the Paper}

 The main results in this paper establish log-concavity estimates for eigenfunction of the problem \eqref{eqn: Weighted eigenfunction equation}.  To begin, let us first consider the case where the potential function $V$ vanishes.

\begin{theorem}\label{thm: log-concavity one-point}
    Suppose that $\Omega \subset \mathbb{M}^n_K$ (for $K \geq 0$) is a convex domain.
    If $\rho:\Omega \to \mathbb{R}$ is a positive function which satisfies
    \begin{equation}\label{eq: simple example}
         \nabla^2 \rho \leq 2 K \rho g_{\mathbb M^n_K},
    \end{equation}
then the principle eigenfunction of \eqref{eqn: Weighted eigenfunction equation} (with vanishing potential) is log-concave with respect to the Levi-Civita connection on $\mathbb{M}^n_K$.
\end{theorem}
In this theorem and throughout the rest of the paper, we denote $\mathbb M^n_K$ to be the simply connected space-form of dimension $n$ with constant sectional curvature $K.$ In other words, $\mathbb M^n_K$ is a sphere when $K>0$, is Euclidean space when $K=0$, and is hyperbolic space when $K<0$.

 The strategy used in the proof of Theorem \ref{thm: log-concavity one-point} is a continuity method
 (see e.g. \cite{singer1985estimate, Caffarelli-Friedman}). In other words, we consider a one-parameter family of PDEs so that the eigenfunction at the initial time is log-concave and the final time is the problem of interest. If the eigenfunction is not log-concave for $t = 1,$ then there must be an intermediate time $t_0$ for which the log-concavity fails. Applying then a maximum principle, we derive a contradiction. We should note that Lemma \ref{mainpropposition} generalizes the main lemma from our previous work in \cite{surfacepaper1}. This theorem implies the following fundamental gap estimate (see Lemma \ref{Andrews-Ni bound}).
\begin{cor}\label{Gap-estimate}
Under the assumption of Theorem \ref{thm: log-concavity one-point},
    the fundamental gap of the problem \eqref{eqn: Weighted eigenfunction equation} satisfies
    \[\Gamma(\Omega) \geq \frac{1}{\|\rho\|_\infty}\left(\frac{\pi^2}{ D^2}+\frac{K}{2}\right)\].
\end{cor}
Note that Theorem \ref{thm: log-concavity one-point} only holds if $K \geq 0$. However, by applying this result for a particular choice of conformal deformation, we prove Theorem \ref{Fundamental gap of Horoconvex-domains}.
 This result is the first known lower bound on the fundamental gap in a space of negative curvature in terms of the diameter.\footnote{There are estimates on the fundamental gap in general settings \cite{ramos2023integral}, but these results require that the domain satisfy an interior rolling ball condition and use the radius of this ball in their gap estimate.} In fact, it is possible to relax the assumption of horoconvexity somewhat (see Subsection \ref{Discussion-on-horoconvexity} for details).

It is possible to generalize Theorem \ref{thm: log-concavity one-point} to Schr\"odinger operators with non-vanishing potential (see Theorem \ref{Log concavity estimate: precise version} for the precise version). From this, one can establish fundamental gap estimates for $C^4$ conformal deformations of a round sphere for domains whose second fundamental form $\tilde h_{\partial \Omega}$ is larger than some constant.

\begin{cor}
     \label{cor: conformal nearly-round spheres}
    Let $(\mathbb{M}^n_K, \tilde g = e^{2\varphi} g_{\mathbb M^n_K})$ be a conformal deformation of a round sphere. There exists $\varepsilon = \varepsilon(n,K)>0$ 
    such that whenever $\|\varphi \|_{C^4}\leq \varepsilon(K,n),$ 
    we have for any $\Omega$ convex with respect to $g_{\mathbb M^n_K},$ the function $ u = e^{\tfrac{n-2}{2}\varphi}\tilde  u$ is log-concave with respect to $g_{\mathbb M^n_K}$ Levi-Civita connection, where $\tilde u$ is the principle eigenfunction of $\Omega$ with respect to $\Delta_{\tilde g}$. Furthermore, the fundamental gap of the domain satisfies
    \begin{equation}
        \Gamma(\Omega, \Delta_{\tilde g} ) \geq \frac{\min \exp(2\varphi)}{\max\exp(2 \varphi)}\frac{\pi^2}{D^2 }+ \frac{1}{\max\exp(2\varphi)}\frac{K}{2}.
    \end{equation}
\end{cor}

\begin{remark}
In the previous corollary, we assumed that the domain is convex with respect to spherical geometry. If one wants to express this condition in terms of the conformal metric $\tilde g$, it suffices to assume to that the second fundamental for on $\partial \Omega$ (with respect to $\tilde g)$ satisfies  $\tilde h_{\partial \Omega}\geq \|\tilde\nabla \varphi\|_{\infty, \tilde g}.$ See Remark \ref{remarks-conformal-theorem}.
\end{remark}

Although there is a large body of work studying the fundamental gap for domains in spherical or flat geometry, there are comparatively few results 
known when the sectional curvature is positive but non-constant. In particular, Corollary \ref{cor: conformal nearly-round spheres} is the first lower bound for domains in positively curved geometries of dimensions three or higher.

The idea of considering a conformal deformation to find a more tractable connection can be applied to other elliptic equations as well. As a simple demonstration of this approach, we consider the torsion problem \begin{equation} \label{Torsion problem}
        \begin{cases}
            \Delta u + 1 =0, \quad x \in \Omega \\
            u(x) = 0, \quad x \in \partial \Omega
        \end{cases}
        \end{equation}
        on a convex domain $\Omega \subset \mathbb{S}^2$. Applying stereographic projection to obtain a flat connection, we can apply a result of Kennington to obtain the following \cite{kennington1985power}.
 \begin{theorem} \label{Corollary to the torsion problem}
    Let $\Omega \subset \mathbb{S}^2$ be a convex domain whose circumradius is at most $2 \arctan(1/5)$ and consider the solution $u$ to the torsion problem.
    Then it is possible to find a point $p \in \Omega$ so that $(\pi_p u)^\frac{1}{3}$ is concave where $\pi_p$ is stereographic projection based at $p$. As a result, the level sets of $u$ must be connected and we can derive lower bounds on their geodesic curvature.
\end{theorem}

\begin{remark}
    This result can be generalized for $C^2$ small deformations of round spheres in a straightforward way.
\end{remark}

\subsection{Structure of the paper}

In Section \ref{Background Section}, we recall some well-known facts about conformal geometry and fundamental gap estimates. We then discuss how to obtain fundamental gap estimates for the problem \eqref{eqn: Weighted eigenfunction equation}, which is done by a comparison argument to the case where $\rho$ is constant. In addition, we show how changing the connection allows us to prove Theorem \ref{Corollary to the torsion problem} 
 In Section \ref{Barrier PDE section}, we prove a generalization of the the barrier method from \cite{surfacepaper1}. Using this, we prove Theorem \ref{thm: log-concavity one-point} in Section \ref{Section: Application}. In this section, we also prove
Theorem \ref{Fundamental gap of Horoconvex-domains}.  

\subsection*{Acknowledgements}

Malik T. wishes to thank Rugang Ye for fruitful discussions. He wishes to thank especially Guofang Wei for several helpful comments that improved this manuscript.
Gabe K. would like to thank Mizan Khan for his helpful suggestions about exposition. 
The authors would also like to thank Xuan Hien Nguyen for her helpful comments.

\section{Background and Preliminaries}\label{Background Section}

\subsection{Conformal deformations of the Laplace operator}

In this section, we recall some well-known facts about conformal geometry.
Given a function $\varphi: M\rightarrow \mathbb R,$ we consider the conformal metric  $\tilde g = e^{2\varphi}g.$
There is a known formula for conformal deformations of Hessian and the Laplace operator, which states that for a smooth function $F:M \rightarrow \mathbb R$
\begin{align}
\Hess_{\tilde g} F &= \Hess_g F - 2 d\varphi \otimes d F +(\nabla \varphi \cdot \nabla F) g \label{eqn: conformal change of Hessian} \\
\label{eqn: conformal change laplacian formula}
 \Delta _{\tilde g}F &= e^{-2\varphi}\Bigl(\Delta_g F +(n-2)\nabla \varphi\cdot \nabla F\Bigr).
\end{align}
Then, we consider a (possibly weighted) $\tilde g$-eigenfunction $\psi$ satisfying
\begin{align}\label{conformal eigenfunction equation}
    \Delta_{\tilde g} \psi = -\lambda {\tilde \rho} \psi  \quad \textup{in }\Omega \quad \textup{and }\quad  \psi = 0 \quad \textup{on }  \partial \Omega.
\end{align}
Observe that for any $a\in \mathbb R$ that 
\begin{align*}
  &\Delta_g \psi +(n-2)\nabla \varphi\cdot \nabla \psi \\
  =&e^{-a\varphi }\Bigl(\Delta_g (\psi e^{a \varphi}) +\Bigl[ - a^2  |\nabla \varphi |^2-a  \Delta_g \varphi \Bigr]\psi e^{a \varphi}\Bigr) -(n-2-2a)\nabla \varphi \cdot \nabla \psi.
\end{align*}
Choosing $a=\frac{n-2}{2}$, we see that the function $u = \psi e^{\tfrac{n-2}{2}\varphi}$ is a weighted eigenfunction of a Schr\"odinger operator of $g$.
\begin{equation}\label{eqn: Conformal eigenfunction equation, Schrodinger form}
      \begin{cases}
       -\Delta_g u +\Bigl[ \tfrac{(n-2)^2}{4}  |\nabla \varphi |^2+\tfrac{n-2}{2}  \Delta_g \varphi \Bigr]u = \lambda   e^{2\varphi} {\tilde \rho} u\\
      u \vert_{\partial \Omega} = 0
   \end{cases},
\end{equation}
which is a composite membrane equation with the weighting function $\displaystyle \rho = {\tilde \rho} e^{2\varphi}$ and potential $V =  \tfrac{(n-2)^2}{4}  |\nabla \varphi |^2+\tfrac{n-2}{2}  \Delta_g \varphi$. Note that for surfaces ($n=2$), when there is no potential in the eigenvalue equation, after a conformal change there is again no potential and the factor $e^{\tfrac{n-2}{2} \varphi}$ is trivial. As a result, changing the metric conformally for surfaces essentially just changes the weighting function. However, in higher-dimensions, a conformal transformation will transform the Laplace operator into a Schr\"odinger operator with a non-vanishing potential term. 

Since convexity of domains will play a central role in this paper, we conclude this subsection by recalling how the second fundamental form of a hypersurface changes under conformal deformation. Given a domain $\Omega$, let $h$ and $\tilde h$ denote the second fundamental form of $\partial \Omega$ under $g$ and $\tilde g,$ respectively. Then for $p \in \partial \Omega$ and tangent vectors $X,Y\in T_p\partial\Omega$ one has 
\begin{equation} \label{Second fundamental form under conformal change}
    \tilde h(X,Y)  = e^{\varphi}\left(h(X,Y)+g(X,Y)\frac{\partial \varphi}{\partial N}\right), 
\end{equation}
where $N$ is the normal vector (unit length with respect to $g$) at $p.$ 
 Thus the smallest principle curvature $\kappa_{\min}$ satisfies
 \begin{equation}\label{principle-curvatures-higher-dimensions.}
     \tilde \kappa_{\textup{min}}= \inf_{\tilde g(X,X) = 1 } \tilde h(X,X)\nonumber = e^{-\varphi}\left(\kappa_{\min}+\frac{\partial \varphi}{\partial N}\right).
 \end{equation}

\subsection{The fundamental gap in hyperbolic space}

\label{Fundamental Gap in hyperbolic space subsection}

As mentioned in the introduction, the fundamental gap  of convex domains in hyperbolic space behaves very differently than in Euclidean or spherical geometries. As for the eigenfunctions of such regions, Shih \cite{shih1989counterexample} constructed convex domains in hyperbolic space whose principle eigenfunction has two distinct maximum points. As a result, such a function cannot be log-concave with respect to any connection since some of its level sets are \emph{disconnected}. Using a similar construction, Bourni et al. constructed convex domains in hyperbolic space (with arbitrary diameter) whose fundamental gap is arbitrarily small \cite{bourni2021explicit,bourni2022vanishing}. In recent work \cite{khan2022negative}, the first named author and Nguyen extended this argument to manifolds with any negative sectional curvature. In particular, for any manifold with even a single tangent plane of negative curvature, there are (small) convex domains whose fundamental gap is arbitrarily small. 

Despite these results, there are a number of open questions remaining about the fundamental gap of domains in negative curvature. In particular, if we strengthen the notion of convexity, we can ask whether it is possible to establish lower bounds on the fundamental gap. For instance, we can consider domains which are \emph{horoconvex}.

\begin{definition}\label{Def: horoconvexity}
A horocycle is a continuous curve in hyperbolic space 
 whose normal geodesics all converge asymptotically in the same direction. Such curves have geodesic curvature identically $1$. We say that a domain $\Omega \subset \mathbb H^2$ is horoconvex if at every point $p \in \partial \Omega$ there exists a horocycle passing through $p$ such that $\Omega$ is contained in the region bounded by the horocycle.
\end{definition}
For the Poincar\'e disk model, horocycles are Euclidean circles entirely contained in the unit disk $B$ and tangent to $\partial B.$ Concerning the fundamental gap of horoconvex domains, Nguyen, Stancu and Wei recently obtained the following result.

\begin{theorem}[Theorem 1.1 \cite{nguyen2022fundamental}] For every $n \geq 2$, there exists a constant $C(n)$ such that the Dirichlet fundamental gap of every horoconvex domain $\Omega$ with diameter $D \geq 4 \ln 2$ satisfies
$$
\Gamma(\Omega) \leq \frac{C(n)}{D^3}.
$$
\end{theorem}
As $D \rightarrow \infty$, the quantity $\Gamma(\Omega) D(\Omega)^2$ tends to $0$, so this shows the fundamental gap of large horoconvex domains is small (relative the gaps of Euclidean domains). Nonetheless, Nguyen et. al. conjectured that it may be possible to recover some lower bound on the gap in terms of the diameter and the dimension. 

To this end, Theorem \ref{Fundamental gap of Horoconvex-domains} partially establishes this conjecture for $n=2$, with the additional assumption that the diameter is sufficiently small. We also note recent work of Grossi and Provenzano, which showed that for any horoconvex domains in $\mathbb{H}^2$, the principle eigenfunction has a unique non-degenerate critical point \cite{grossi2023critical}. This result gives some partial evidence for log-concavity, since it implies that the level sets must be connected.

\subsection{Spectral Gap Estimates}

We now turn our attention to proving fundamental gap estimates for the problem of the form \eqref{eqn: Weighted eigenfunction equation}. Our strategy to obtain fundamental gap estimates uses an import insight of \cite{singer1985estimate}.
The function $w = \tfrac{u_2}{u_1},$ the ratio of the second and first eigenfunction, satisfies a PDE with Neumann boundary conditions: 
\begin{lemma}[\cite{singer1985estimate}]\label{Ratio of eigenfunctions lemma}
    Consider the function $w = \frac{u_2}{u_1}$. Then $w$ satisfies the equation
    \begin{equation} \label{The gap is the eigenvalue}
         \Delta w +2\nabla \log u_1\cdot \nabla w=-\Gamma\rho w
    \end{equation}
    with Neumann boundary conditions.
\end{lemma}
Hence, estimating the fundamental gap can be reduced to estimating the  first non-trivial weighted Neumann eigenvalue of the operator $-\Delta_g -2 \nabla \log u_1 \cdot \nabla   (\cdot).$ To estimate these, we make use of a result by Andrews and Ni \cite{andrews2012eigenvalue}.
\begin{lemma}[Proposition 3.1 \cite{andrews2012eigenvalue}] \label{Andrews-Ni bound}
Let $\Omega$ be a convex domain in any Riemannian manifold with $\mathrm{Ric}_{i j}+f_{i j} \geq a g_{i j}$ for some $a\geq 0$. Then the second Neumann eigenvalue of the problem
 \[ -\Delta_g u + \langle \nabla u, \nabla f \rangle+ \mu \rho u =0  \]
satisfies the estimate
\[\|\rho\|_\infty \mu_2(\rho) \geq \frac{a}{2}+\frac{\pi^2}{D^2} \]
\end{lemma}
\begin{proof}
    Note that the second eigenvalue satisfies 
    \begin{align*}
        \mu_2(\rho) &=
   \min _{V\subset H^1, \, \textup{dim}V=2}\max _{u \in V} \frac{\int_{\Omega}|\nabla u|^2 e^{-f}\, dx}{\int_\Omega \rho u^2 e^{-f}\, dx}\\
   &\geq\frac{1}{\|\rho\|_\infty} \min _{V\subset H^1, \, \textup{dim}V=2}\max _{u \in V} \frac{\int_{\Omega}|\nabla u|^2 e^{-f}\, dx}{\int_\Omega u^2 e^{-f}\, dx}  = \frac{\mu_2(1)}{\|\rho\|_\infty}.
    \end{align*}
    The conclusion then follows from \cite{andrews2012eigenvalue}.
\end{proof}
In the case of the fundamental gap problem, we choose $ f = -2 \log u_1.$ Hence, when the Ricci curvature is non-negative, all we need to show is that the first eigenfunction is log-concave.

\subsection{A convexity result for the torsion problem}

The results in this paper fall into a broader class of convexity and quasi-convexity results for the solutions to uniformly elliptic PDEs. There are a large number of such results (see, e.g., \cite{makar1971solution,korevaar1983convex,guan2005convex,steinerberger2022concavity,jia2023remarks} and the references therein). Although the primary focus of this paper is eigenvalue problems, conformal geometry can be used to study other elliptic equations as well. As an example, we demonstrate how conformal change can can be used to obtain convexity results for the torsion problem. 
In 1971, Makar-Limanov \cite{makar1971solution} showed that for convex $\Omega\subset \mathbb 
R^2,$ the solution to \eqref{Torsion problem} is $1/2$-concave (see \cite{ma2012convexity} for a discussion of the higher dimensional problem). However, if one tries to adapt this result to more general geometries, a large number of curvature terms appear in the computation which complicates the analysis. 
In \cite{Korevaar1987}, Korevaar mentions unpublished work of himself and Treisberg which performs this analysis and shows that solutions to the torsion problem are $1/2$-concave in spherical geometry.\footnote{Recent work of Grossi and Provenzano \cite{grossi2023critical} implies that the level sets of the solutions must be connected whenever $\Omega$ is a convex domains in spherical geometry or a horoconvex domain in hyperbolic space.} However, if one uses conformal deformations, we can derive Theorem \ref{Corollary to the torsion problem} almost immediately from \cite{kennington1985power}. 

 \begin{proof}[Proof of Theorem \ref{Corollary to the torsion problem}]
     We let $\beta \geq 1$ be a constant (the original claim follows by choosing $\beta = 1$). The bound of the circumradius allows us to find a rotation of the sphere so that $\Omega$ contains the south pole and such that $\Omega$ is contained in the ball of radius $2\arctan(\tfrac{1}{1+4\beta})$ around the south pole. 
     Since the spherical metric is conformal to the Euclidean metric,
     \begin{align*}
         g_{\mathbb S^2} = \frac{4}{(1+\|x\|^2)^2}g_{\mathbb R^2}.
     \end{align*}
     We then rewrite the torsion problem in terms of stereographic projection from the north pole. Doing so, in view of \eqref{eqn: conformal change laplacian formula}, \eqref{Torsion problem} becomes
     \begin{equation}\label{Torsion problem on sphere}
         \Delta_{\mathbb{R}^2} u  + \frac{4}{(1+\|x\|^2)^2} = 0.
     \end{equation}
Since $\Omega$ is convex as a spherical domain and contains the north pole, an exercise in spherical geometry shows that the image of $\Omega$ under stereographic projection is a convex set. As such, whenever $\rho = \frac{4}{(1+\|x\|^2)^2}$ is a $\beta$-concave function, i.e. $\rho^\beta$ is concave, we appeal to a result by Kennington \cite{kennington1985power} to see that $u^{\tfrac {\beta}{1+2\beta}}$ is a concave function. 
The eigenvalues of the Hessian of $\rho^\beta$ are  \[ -\frac{4^{1+\beta}\beta}{(1 +\|x\|^2)^{1+2\beta}} \quad \textrm{ and } \quad \frac{4^{1+\beta}\beta(-1 + (1+4\beta) \|x\|^2)}{(1 +\|x\|^2)^{2(1+\beta)}}, \] both of which are negative whenever $\|x\|^2<1/(1+4\beta)$. By the assumption on the circumradius, the image of $\Omega$ is contained within this disk and so $u^{\tfrac{\beta}{1+2\beta}}$ is a concave function in stereographic coordinates. Letting $\beta =1,$ the claim follows.
\end{proof}

Note that this result also implies lower bounds on the geodesic curvature of the level sets of $u$. In particular, the level sets of $(\pi_p u)^\frac{1}{3}$ will have non-negative curvature, so we can bound the second fundamental form of the level sets of $u$ in terms of the conformal factor using \eqref{Second fundamental form under conformal change}. Moreover, if we further restrict the circumradius, it is possible to prove stronger convexity assumptions. More precisely, if we assume that the circumradius is at most $2 \arctan( 1/(1+4\beta))$, we can repeat the argument to show that $u^{\tfrac{\beta}{1+2\beta}}$ is concave under stereographic projection. 
Finally, we note that this argument does not require for the geometry to be perfectly spherical (in contrast with the previous known results on this problem), and can easily be generalized to consider $C^2$ deformations of a round $\mathbb{S}^2$ by modifying the weighting function.\footnote{An immediate consequence of the Uniformization Theorem is that a generic perturbation of the Riemann sphere can be written as a conformal deformation.}

\section{Log-concavity of the First Eigenfunction}

\subsection{Log-concavity via the barrier PDE approach}
\label{Barrier PDE section}

In \cite{surfacepaper1}, the first and third named authors (along with Nguyen and Wei) provided a systematic way to apply the continuity method for establishing log-concavity estimates by constructing barrier functions which satisfy a particular differential inequality. In this paper, we present a generalization of this method.

Let us first set some notation and definitions. For a Riemannian manifold $(M^n,g)$, we write the $(1,3)$ curvature tensor of its Levi-Civita connection as \[
R(X,Y)Z = \nabla_X \nabla_Y Z -  \nabla_Y \nabla_X Z - \nabla_{[X,Y]} Z \]
and define $R_X$ to be the $(1,1)$-tensor given by \[
R_X (Y) = R(Y,X) X. 
\]

\begin{definition} \label{Barrier operator definition}
    Given $V, \rho:\Omega\rightarrow \mathbb R$, we let $v = \log u_1$ be the logarithm of the first eigenfunction of the problem \eqref{eqn: Weighted eigenfunction equation}. Then, for any function $b: \Omega \rightarrow \mathbb R$ and a unit vector $X$ (with respect to some metric $g$), the barrier operator $\mathcal{B}$ is the quantity
\begin{align} \label{Barrier operator definition equation}
    \mathcal{B}(b,X) :=& -2b^2 +2\langle \nabla b, \nabla v\rangle -2\textup{tr}\left( R_{X}\circ (\nabla v \otimes \nabla v + \textup{Hess }v)\right)
    \\
&\quad -\nabla _{\nabla v}\Ric(X,X)+2\nabla _{X}\Ric(X, \nabla v)+\Delta b(p) - \lambda\rho_{XX} + V_{XX}, \nonumber
\end{align}
where $\lambda$ denotes the principle eigenvalue of Equation \ref{eqn: Weighted eigenfunction equation} and the curvature and derivative terms are taken with respect to the Levi-Civita connection of $g$.
\end{definition}

Note that this operator is the same as the one in \cite{surfacepaper1}, except for the final two terms which are induced by the variable density and the potential, respectively. This operator appears in a certain maximum principle computation and so we define a condition known as the \emph{barrier criteria}.

\begin{definition} \label{Barrier criteria definition}
A barrier function $b$ satisfies the \emph{barrier criteria} if $\mathcal{B}(b,X) >0$
    whenever $X \in U \Omega$ is a unit vector\footnote{Here, $U\Omega$ is the unit tangent bundle of $\Omega$.} such that the mapping
\begin{align}
 U\Omega \rightarrow \mathbb R, \quad X_q\mapsto  \Hess \,v(X_q,X_q) +b(q)  \label{map}
\end{align} 
achieves a maximum at $X$ and satisfies $\Hess\,v (X,X) + b(p)=0$. 
\end{definition}

As we shall see, the barrier criteria prevents the Hessian of eigenfunction from ever \emph{touching} $b$. In particular, along a one-parameter family of eigenvalue problems, it is impossible for there to be a first time and an interior point $q$ so that  $\Hess \,v(X_q,X_q) = b$. More precisely, we have the following.

\begin{lemma}\label{mainpropposition}
Let $M^n$ be a smooth manifold and consider a one-parameter family of eigenvalue problems
\begin{equation}\label{smooth-family-of-PDEs}
       -\Delta_{g(t)} \varphi + V(t) \varphi = \lambda \rho(t) \varphi,  \quad
       \varphi \vert_{\partial \Omega(t)} \equiv 0
\end{equation}
Here, the metric $g(t)$, the domain $\Omega(t)$, the weighting $\rho(t)$ and the potential $V(t)$ are all allowed to depend on $t$, so long as we assume that the domain $\Omega(t)$ is geodesically convex with respect to the Levi-Civita connection of $g(t)$ for all $t$.
Suppose there is a function $b:\Omega(t) \to \mathbb{R}$ which satisfies the following assumptions:
\begin{enumerate}
    \item $b$ depends smoothly on both $x$ and $t$.
    \item The barrier $b(x,t)$ is uniformly bounded in $t$. \label{Bounded near the boundary assumption}
\item At time $t=0$, \begin{equation} \label{Initial time assumption}
    \Hess \, v(x,0) + b(x,0)\,  g(0) <  0.
\end{equation}
\item And finally, for all $0 \leq t \leq 1$, $b$
satisfies the barrier criteria.
\end{enumerate} 

Then the function $v(x,1)$ satisfies the concavity estimate 
\begin{equation} \Hess \, v(x,1) + b(x,1)\,  g \le  0   \label{concave v} \end{equation} 
on the original domain $\Omega(1)$. 
\end{lemma}

\begin{proof}[Proof of Lemma \ref{mainpropposition}]

By Assumption \ref{Initial time assumption}, we have that \[ \Hess_{g(0)} \, v(x,0) + b_0(x)\,  g(0) <  0. \]
 For the sake of contradiction, suppose that $\Hess_{g(1)} \, v(X,X)_p+b(p)g(1)>0$ for some unit vector $X\in T_p\Omega(1),$ $p\in \Omega(1).$  Using the continuity in $t$ and the fact that the estimate cannot fail at the boundary,
there must be a time $t_0 \in (0,1)$ after which the Hessian bound fails to hold. By Assumption \ref{Bounded near the boundary assumption}, the Hessian bound holds in a neighborhood of the boundary for all time (c.f. Lemma 3.4 of \cite{10.4310/jdg/1559786428}). As such, in order for the inequality to fail there must be an interior point $p$ and a unit vector $X_p\in T_p\Omega(t_0)$ such that 
\begin{align*}
  0= \Hess_{g(t_0)} \, v(X_p,X_p) +b(p)=\max_{Y_q \in U \Omega(t_0) } \left(\Hess_{g(t_0)} \, v(Y_q,Y_q) +b(q) \right).
\end{align*}
 We denote $e_1=X_p$, 
and extend this vector to an orthonormal basis $\{e_j\}$ of $T_p\Omega(t_0)$. 
Let us denote $e_1=X_p$,  Since $e_1$ is the maximal direction of $\textup{Hess}\, v$ at $p$, which is symmetric, so $e_1$ is an eigenvector of $\textup{Hess}\, v$. Since $\{e_j\}$ is an orthonormal basis, we have at $p$, 
\begin{align}
    \textup{Hess}\,v_p(e_1,e_j)= \langle \textup{Hess}\,v_p(e_1),\ e_j \rangle  =0\quad \textup{ for }j =2, \dots, n. \label{maximum-direction-property} 
\end{align}
where we here and in the following write $\Hess$ instead of $\Hess_{g(t_0)}.$
By the maximum principle, we find that
\begin{align}
   e_i(v_{11})(p)+e_i(b)(p)&=0\quad \textup{for all }i =1,\dots, n  \label{gradient-zero} \\
   \Delta ( v_{11} ) (p)+\Delta b(p)&\leq 0.
\end{align}
We now compute $\Delta ( v_{11} ) (p)+\Delta b(p)$. Commuting the indices, we have that 
 \begin{align*}   \sum_iv_{11,ii}=\sum_iv_{ii,11}+2v_{11}\Ric_{11}-2\sum_{ij}v_{ji}R_{1j1i}+\sum_{i}2v_i\Ric_{i1,1}-\sum_iv_i\Ric_{11,i}.
\end{align*}
To compute the right hand side, 
we note that $v$ satisfies the following equation 
\begin{align}\label{vsPDE}
    \|\nabla v\|^2=V-\lambda\rho -\Delta v,
\end{align}
Taking the derivative of both sides of \eqref{vsPDE}, one has that 
\begin{align*}
    \sum_{i}v_{ii,11}+ \lambda \rho_{11} - V_{11}=-2\sum_{i}(v_{i1}^2+v_{i1,1}v_i). 
\end{align*}
This gives that when commuting the indices
\begin{align*}
   v_{i1,1} =v_{11,i}+\sum_jv_jR_{j1i1} 
   =-b_i+\sum_jv_jR_{j1i1}.
\end{align*}
Hence,    
\begin{align*} \sum_{i}v_{ii,11}&=-2\sum_{i}\left(v_{i1}^2-v_ib_i+\sum_j v_iv_jR_{j1i1}\right)\\ &=-2b^2+2\langle \nabla b, \nabla v\rangle -2\sum_{i,j} v_iv_jR_{j1i1}.
\end{align*}

Putting this all together, we find that
\begin{eqnarray*}
0&\geq& \Delta (v_{11}+b)(p)\\
&= & -2b^2+2\langle \nabla b, \nabla v\rangle -2\sum_{i,j} v_iv_jR_{j1i1}-2\sum_{i,j}v_{ij}R_{1j1i}  \\
& &-2b\Ric_{11}-\sum_{j}v_j\Ric_{11,j}+2\sum_{j}v_j\Ric_{1j,1}+\Delta b(p)- \lambda\rho_{11}+ V_{11}\\
&=&-2b^2+2\langle \nabla b, \nabla v\rangle -2\sum_{i,j} R_{j1i1}(v_iv_j+v_{ij})
-\nabla _{\nabla v}\Ric(e_1,e_1) \\
& &+2\nabla _{e_1}\Ric(e_1, \nabla v)+\Delta b(p) - \lambda\rho_{11} + V_{11}
\end{eqnarray*}
which contradicts the fact that $\mathcal{B}(b)>0$.
\end{proof}

\section{Applications of Conformal Deformations}\label{Section: Application}

We now provide a number of application for the previous approach. We start with a simple example, which is to prove Theorem \ref{thm: log-concavity one-point}. In fact, we will prove a generalized version where the potential $V$ may be non-vanishing. Since our goal is to show that the principle eigenfunction is log-concave, we use the barrier $b=0$ and construct a continuity family to apply Lemma \ref{mainpropposition}.

\begin{theorem}  \label{Log concavity estimate: precise version}
    Suppose that $\Omega\subset \mathbb M^n_K$ is a convex domain. Furthermore, suppose that $V$ and $\rho$ satisfy the inequality
    \begin{align}\label{condition-Theorem-log-concavity}
         \nabla^2 (V-\lambda(t) \rho) >  2K(V-\lambda(t) \rho)
     \end{align}
     for all $t\in (0,1)$ where $\lambda(t)$ is the principle eigenvalue of the problem \eqref{smooth-family-of-PDEs} with potential $V(t) = tV$ and weighting function $\rho(t) = t\rho+(1-t).$
    Then the principle Dirichlet eigenfunction of the problem \eqref{eqn: Weighted eigenfunction equation} is log-concave.
\end{theorem}

\begin{proof}[Proof of Theorem \ref{thm: log-concavity one-point}]

We fix the domain $\Omega$ and the barrier
 $b(x,t)\equiv 0$. To apply Lemma \ref{mainpropposition}, we consider $\rho(t) = (1-t) + t \rho$ and the potential $V(t) =  t V$. In $\mathbb M^n_K$, the barrier operator becomes
\begin{align*}
    \mathcal{B}(0) &=-2\sum_{i=2}^nK(v_i^2+v_{ii})-\lambda\left((1-t) + t \rho \right)_{11}+tV_{11}\\
    &=-2K(\Delta v+|\nabla v|^2-v_{11}-v_1^2)-\lambda t \rho_{11} + t V_{11}\\
    &=-2K(tV-\lambda \left((1-t) + t \rho \right) -v_1^2)-\lambda t \rho _{11}+tV_{11},
\end{align*}
where we used \eqref{vsPDE} in the last equality. 
In the case where $V$ and $\rho$ satisfy the first set of assumptions, these hypotheses imply that 
\begin{equation} \label{Positivity of barrier equal 0}
\mathcal{B}(0)=t\lambda(2K\rho-\rho_{11})+2Kv_1^2+2K\lambda(1-t) +t(V_{11} - 2KV)>0\end{equation}
for any $0<t<1$. We thus conclude the claim from Lemma \ref{mainpropposition}.
\end{proof}

Similarly, this approach can be used to prove log concavity properties of eigenfunctions under conformal change on $\mathbb M^n_K$ (with $K> 0$).

\begin{proof}[Proof of Corollary \ref{cor: conformal nearly-round spheres}]
 
By our assumption, $\Omega$ is convex with respect to $g.$ 
Recall from the calculation \eqref{eqn: Conformal eigenfunction equation, Schrodinger form} $ u  = \tilde u \exp(\tfrac{n-2}{2}\varphi)$ satisfies the equation \begin{equation*}      \begin{cases}
       -\Delta_{\mathbb M^n_K} u +\Bigl[ \tfrac{(n-2)^2}{4}  |\nabla \varphi |^2+\tfrac{n-2}{2}  \Delta \varphi \Bigr]u = \lambda   e^{2\varphi}  u\\
      u \vert_{\partial \Omega} = 0
   \end{cases},
\end{equation*}
where $ \tilde u$ is the first Dirichlet eigenfunction on $\Omega$ with respect to $\tilde g.$

To apply Theorem \ref{Log concavity estimate: precise version}, we let $ V = \tfrac{(n-2)^2}{4}  |\nabla \varphi |^2+\tfrac{n-2}{2}  \Delta \varphi$ and $\rho =  \exp({2\varphi})$ we use the same deformation, $ V(t)  = tV$ and $\rho(t) = t \rho + (1-t).$

 It thus suffices to show that  $\nabla^2 (V-\lambda(t) \rho) - 2K(V-\lambda(t) \rho)>0.$
Note that the eigenvalue $\lambda$ is bounded from below. Indeed, using the Raleigh quotient, we get that 
\begin{align*}
    \lambda(\rho)  &= \inf \frac{\int_\Omega |\nabla u|^2+Vu^2\, dx}{\int_\Omega \rho u^2 \, dx}\\
&\geq \frac{1}{\| \rho\|_\infty } \left(\inf \frac{\int_\Omega |\nabla u|^2\, dx}{\int_\Omega  u^2 \, dx}+\min_\Omega V\right)\\
&= \frac{1}{\|\rho\|_\infty}\left(\lambda_1(\Omega)+\min_\Omega V\right),
\end{align*}
where $\lambda_1(\Omega)$ denotes the first Dirichlet eigenvalue of the usual equation $-\Delta u = \lambda u$ in $\Omega$ with Dirichlet boundary conditions.
On the other hand, Ling \cite{ling2006lower} showed that the first eigenvalue satisfies the estimate
\begin{align*}
    \lambda_1(\Omega) \geq \frac{1}{2}(n-1) K+\frac{\pi^2}{R_{\mathbb{M}^n_K}^2}, 
\end{align*}
where $R_{\mathbb M^n_K}$ is the in-radius of the domain. One can now choose $\varepsilon(n,K)>0$ small enough such that \eqref{condition-Theorem-log-concavity} holds true. 

We therefore infer that $v = \log u $ is concave with respect to $g_{\mathbb M^n_K}.$
Applying Lemma \ref{Andrews-Ni bound}, we obtain 
\begin{align*}
    \Gamma(\Omega) \geq \frac{1}{\|\rho\|_\infty}\left(\frac{\pi^2}{D_{\mathbb M^n_K}}+\frac{K}{2}\right),
\end{align*}
where $D_{\mathbb{M}^n_K}$ denotes the diameter of $\Omega$ with respect to metric $g_{\mathbb M^n_K}.$ Since $D_{\mathbb M^n_K} \leq  D_{\tilde g}/\min \exp(\varphi),$ the claim follows.
\end{proof} 

\begin{remark}\label{remarks-conformal-theorem}
\begin{enumerate}
    \item[$i)$] The assumption on the convexity of $\Omega$ with respect to $g_{\mathbb M^n_K}$ can be rephrased in terms of the metric $\tilde g.$ Observe that by \eqref{principle-curvatures-higher-dimensions.} and by the additional assumption that $\tilde h_{\partial \Omega}\geq \|\tilde\nabla \varphi\|_{\infty, \tilde g}$, we get that for the smallest principle curvature at $p,$ written $\kappa_{\min}(p),$ we have that 
\begin{align*}
    e^{-\varphi}\kappa_{\min}(p) = \tilde \kappa_{\min}(p)-e^{-\varphi}\langle \nabla \varphi, N\rangle \geq \tilde \kappa_{\min}(p) - \| \tilde\nabla \varphi \|_{\infty, \tilde g}>0.
\end{align*} 
\item[$ii)$] One of the assumptions of Corollary \ref{cor: conformal nearly-round spheres} is that the deformation is close in the $C^4$ sense. In dimension $n \geq 3,$ we need this assumption since the potential function $V$ must be small in $C^2$ and $V$ depends on two derivatives of the conformal factor. However, in two dimensions it suffices to assume that the metric is only $C^2$ close to a round metric, as the potential vanishes. This differs from our previous results, where we needed a $C^4$-estimate for the deformation. However, there we were able to prove a fundamental gap estimate for all convex domains with respect to $\tilde g$  whereas here we need a different convexity assumption \cite{surfacepaper1, khan2023modulus}.
\item[$iii)$] It is possible to  establish stronger log-concavity estimates by setting the barrier $b$ to be some non-zero constant. For more details, see \cite{surfacepaper1}.
\end{enumerate}
\end{remark}

\subsection{Fundamental Gap Estimates of Horoconvex Domains}
We now turn our attention to proving Theorem \ref{Fundamental gap of Horoconvex-domains} to obtain gap estimates for horoconvex domains. In hyperbolic space, we cannot apply the barrier argument directly since the curvature terms have an unfavorable sign. Instead, the strategy is to make use of the fact that the metric in $\mathbb H^2$ is conformal to a spherical metric, so we can choose a conformal connection with constant positive curvature. By choosing the models of hyperbolic geometry and spherical geometry carefully, we can find a region in spherical geometry where the weighting function is not too convex so that Theorem \ref{thm: log-concavity one-point} and its corollary apply.

\begin{proof}[Proof of Theorem \ref{Fundamental gap of Horoconvex-domains}]

Consider the Poincar\'e disk model of hyperbolic space, where the metric is given by $g_{\mathbb{H}^2} = \frac{4}{(1-\|x\|^2)^2}  g_{\mathbb{R}^2}$.
Without loss of generality, we suppose that $\Omega \subset B_r(0,g_{\mathbb R^2}),$ where $r$ will be specified below.
The eigenvalue equation
\[  \Delta_{\small\mathbb{H}^2} u + \lambda_1 u = 0 \quad \textup{in }\Omega\quad \&\quad
u  = 0 \quad \textup{on } \partial \Omega, \]
using \eqref{eqn: conformal change laplacian formula} becomes
\[  \Delta _{\mathbb R^2} u +  \frac{4}{(1 - \|x\|^2)^2}\lambda u = 0\quad \textup{in }\Omega \quad \&\quad u = 0 \quad \textup{on }\partial \Omega, \]
i.e. this an equation of the form \eqref{eqn: Weighted eigenfunction equation} with  $ \overline\rho =\frac{4}{(1-\|x\|^2)^2}$, $V \equiv 0.$ Note that $\overline \rho$ has positive eigenvalues, so that we cannot apply Theorem \ref{thm: log-concavity one-point} directly in flat geometry. To circumvent this issue, 
we then change the metric conformally again, writing \[g_{\mathbb{R}^2} = \frac{(R^2 +\|x\|^2 )^2}{4 R^4}g_{\mathbb{M}^2_{K}},\] where $K = \tfrac{1}{R^2}$ and $R>0$ is a constant yet to be determined.\footnote{This is the spherical metric (of radius $R$) in stereographic coordinates.
}
In this new metric, using again \eqref{eqn: conformal change laplacian formula}, the equation becomes 
\begin{align*}
    -\Delta _{\mathbb M^2_K}  u = \frac{(R^2+\|x\|^2)^2}{R^4(1-\|x\|^2)^2} \lambda u.
\end{align*} 
 In order to apply Theorem \ref{thm: log-concavity one-point}, we must verify \eqref{eq: simple example} for $\rho (x) = \tfrac{(R^2+\|x\|^2)^2}{R^4(1-\|x\|^2)^2}$ with respect to the Levi-Civita connection of $\mathbb M^n_K$. In addition, we need to verify that $\Omega$ is convex with respect to $g_{\mathbb M^2_K}.$ We therefore divide this into two parts.\\
\textit{Step 1.}\footnote{A derivation of the calculations in this section can be found in the following \href{https://www.wolframcloud.com/obj/gabekhan/Published/Hessian in Spherical Geometry using disk model.nb}{Mathematica notebook}.}
To compute the Hessian of the weighting function $\rho $ we use \eqref{eqn: conformal change of Hessian}. Thus we calculate  
\begin{align*}
    \nabla_{\mathbb R^2} \rho = \frac{4(1+R^2)(R^2+\|x\|^2)}{(1-\|x\|^2)^2}x
\end{align*}
and
\begin{align*}
    \partial _1^2\rho &=  \left(\frac{4 (1 + R^2) (3 x_1^4 + x_2^2 - x_2^4 + R^2 (1 + 5 x_1^2 - x_2^2) + 
   x_1^2 (3 + 2 x_2^2))}{R^4(1-\| x\|^2)^4}\right) ,\\
    \partial _2^2\rho &= \left(\frac{4 (1 + R^2) (-((-1 + x_1^2) (R^2 + x_1^2)) + (3 + 5 R^2 + 2 x_1^2) x_2^2 + 
   3  x_2^4)}{R^4(1-\|x\|^2)^4}\right) ,\\
     \partial _1 \partial _2\rho &= \left(\frac{8 (1 + R^2) x_1 x_2 (1 + 3 R^2 + 2 \|x\|^2)}{R^4(-1 + \|x\|^2)^4}\right),
\end{align*}

Setting \begin{align*}
    \varphi= \log \left(\frac{2 R^2}{R^2 +\|x\|^2 }\right),
\end{align*}
the conformal factor is $\exp(2\varphi)=\tfrac{4 R^4}{(R^2 +\|x\|^2 )^2}.$ It is straightforward to calculate
\begin{align*}
    -2\nabla \rho \otimes \nabla \varphi + (\nabla \rho \cdot \nabla  \varphi) I_2 = \frac{8(1+R^2)}{(1-\|x\|^2)^3}\begin{pmatrix}
        -x_1^2+x_2^2 &-2x_1x_2\\
        -2x_1x_2 & x_1^2-x_2^2
    \end{pmatrix}.
\end{align*}
We thus conclude 
\begin{align*}
    &\Hess_{\mathbb M^2_K} \rho\\
    &= \frac{4 (1 + R^2)}{(1-\|x\|^2)^4}\begin{pmatrix}
        5 x_1^2 - x_2^2 + R^2 (1 + 5 x_1^2 - x_2^2) + \|x\|^4 &
         6(1+R^2)x_1x_2 \\ 6(1+R^2)x_1x_2
         & 5x_2^2-x_1^2+R^2(1+5x_2^2-x_1^2)+\|x\|^4  
    \end{pmatrix}.
\end{align*}
Our goal now is to verify \eqref{eq: simple example}, i.e., to show that 
\begin{align}\label{Hess-condition}
    \Hess_{\mathbb M^2_K} \rho \leq 2 K \rho g_{\mathbb M^2_K}.
\end{align}
The right hand side becomes 
\begin{align*}
    2K\rho g = \frac{2}{R^2}\frac{4}{(1-\|x\|^2)^2}I_2.
\end{align*}
To verify \eqref{Hess-condition}, we calculate the eigenvalues of the left-hand side of \eqref{Hess-condition} and find that 
\begin{align*}
    \mu _1 = \frac{4(1+R^2)(R^2-\|x\|^2)}{R^4(1-\|x\|^2)^3}, \quad \mu_2 = \frac{4(1+R^2)(\|x\|^2  (5 + \|x\|^2) + R^2 (1 + 5\|x\|^2))}{R^4(1-\|x\|^2)^4}.
\end{align*}
Thus \eqref{Hess-condition} is equivalent to 
\begin{align*}
    \max\{\mu_1,\mu_2\} \leq \frac{2}{R^2}\frac{4}{(1-\|x\|^2)^2}.
\end{align*}
 Note also that $\mu_1,\mu_2$ depend only on $r = \|x\|$, so we can rewrite this inequality in terms of $r$. Doing so, the desired inequality becomes
 \begin{align*}
     \max \{\mu_1(r),\mu_2(r)\}\leq \frac{2}{R^2}\frac{4}{(1-r^2)^2}.
 \end{align*}
This is equivalent to
\begin{align*}
  \underbrace{ \frac{(1+R^2)(R^2-r^2)}{R^2(1-r^2)}}_{=:\tilde \mu_1(r)},  \underbrace{\frac{(1+R^2)(r^2  (5 + r^2) + R^2 (1 + 5r^2))}{R^2(1-r^2)^2}}_{=:\tilde \mu_2(r)}\leq 2.
\end{align*}

Taking $R<1$, we note that
\begin{align*}
  \frac{\partial }{\partial r}\frac{R^2-r^2}{1-r^2}= \frac{2r(R^2-1)}{(1-r^2)^2}<0.
\end{align*}
Since, $\tilde \mu_1(0) <2,$ this implies that $\tilde{\mu}_1(r) <2 $ for all $r<1.$ 

Therefore, we must now establish the inequality for the second eigenvalue. We can rearrange the inequality
\begin{align*}
    \frac{(1+R^2)(r^2  (5 + r^2) + R^2 (1 + 5r^2))}{R^2(1-r^2)^2} <2
\end{align*}
as a quadratic in $r^2$. Doing so, we see that the inequality holds whenever
\begin{align*}
    r^2< \frac{-5-5R^4 -14R^2+(1+R^2)\sqrt{25+94R^2+25R^4}}{2-2R^2}.
\end{align*}
The right hand side of this inequality vanishes when $R$ is zero and is maximized for $R = \frac{\sqrt{7-\sqrt{33}}}{2} \approx 0.560232.$ 
As such, we choose this value as the radius of our sphere. Doing so, the inequality holds whenever 
\begin{align*}
    r^2 < 0.0217494, \quad \textup{i.e. } r< 0.147477.
\end{align*}
In other words, the desired inequality holds within a ball of radius $r=0.147477$ with respect to $g_{\mathbb R^2}$. So we must assume that the domain $\Omega$ has hyperbolic circumradius at most $C= 2\textup{arctanh}(r) = 0.297121.$ We then appeal to a result of Dekster \cite{dekster1995jung} to see that any domain of circumradius greater $0.297121$ must have diameter greater than
$ D_{\max{}} = 2\textup{arcsinh}( \sqrt{3}/2 \sinh(C)) \approx 0.516475$. Solving for $D_{\max{}}$ explicitly, we find that for any domain $\Omega_\mathbb{H}$ whose diameter is at most $2 \textup{arccsch}(2 \sqrt{11/3} )$, it is possible to find a M\"obius transformation so that it is contained entirely within the ball $B_r(0, g_{\mathbb R^2})$.

\textit{Step 2.}
We now verify that $\Omega$ is convex with respect to $g_{\mathbb M^2_K}.$ 
To see this, note that the geodesic curvature of $\partial \Omega$ for a horoconvex $\Omega \subset B_r(0, g_{\mathbb R ^2})$ has curvature $\kappa_{\mathbb R^2} \geq \frac{1}{\left(\tfrac 12 +r\right)^2}$ with respect to $\mathbb R^2.$ Indeed, for each $p\in \partial \Omega,$ there exists a horocycle tangent to $\Omega,$ and $\Omega$ must be contained within that horocycle. Hence, the geodesic curvature of $\partial \Omega$ is greater than or equal to the curvature of that horocycle. To show that the geodesic curvatures $\kappa_{\mathbb M^2_K}$ are positive, we use \eqref{principle-curvatures-higher-dimensions.}: 
\begin{align*}
    \kappa_{\mathbb M^2_K} &\geq  e^{-\varphi}\left(\kappa_{\mathbb R^2}-\|\nabla \varphi\|_{\infty}\right)\\
    & \geq e^{-\varphi}\left(\frac{1}{(\tfrac 12+r)^2}-\frac{\|x\|}{R^2+\|x\|^2}\right)\\
    &\geq e^{-\varphi}\left(\frac{1}{(\tfrac 12+r)^2}-\frac{r}{R^2+r^2}\right)\\
    &\geq e^{-\varphi}\left(\frac{1-2r}{(\tfrac 12+r)^2}\right) >0,
\end{align*}
where in the first inequality, we used Cauchy-Schwarz.
In the second inequality, we used the fact that the Euclidean curvature of horocycles is greater or equal than $1/(\tfrac 12+r)^2.$ In the third inequality, we used that the monotonicity of the function $s \mapsto \tfrac{s}{R
^2+s^2}$ for small $s.$ In the fourth inequality, we used that $R>\tfrac{1}{2}$ and that $(\tfrac 12+ r)^2 \leq 2(\tfrac14+r^2).$ This finishes Step 2.\footnote{For a demonstration of how the horocycles contain the spherical geodesics, we have written a  \href{https://www.geogebra.org/calculator/p8s7f9eg}{Geogebra notebook}.}

\begin{figure}
    \centering
    \includegraphics[width=.8\linewidth]{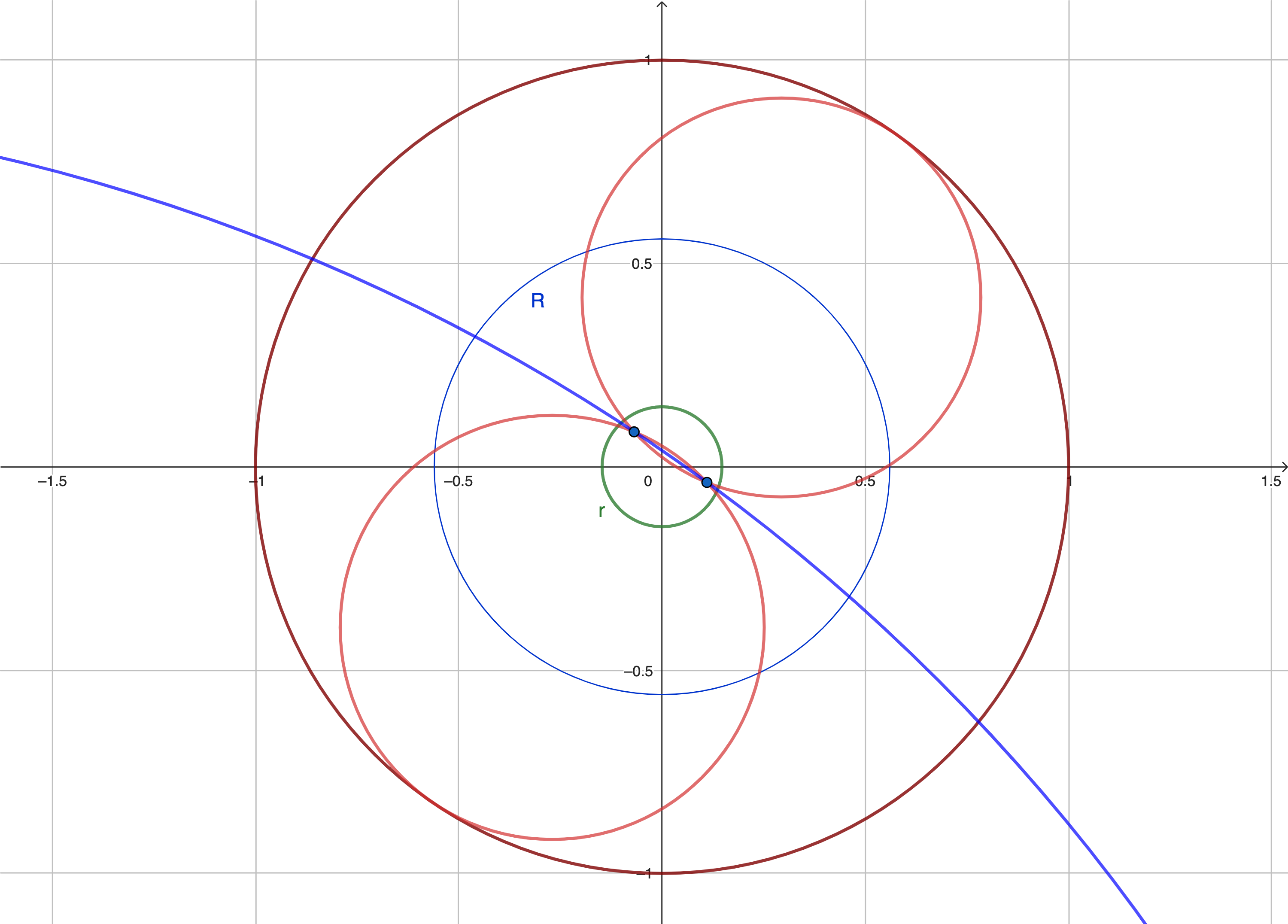}
    \caption{In this figure, the blue spherical geodesic between the two points is contained within the two red horocycles. The previous argument shows that this is always the case when the points are contained within the green circle}
    \label{fig:Horocycle lemma}
\end{figure}

  In view of Step 1 and Step 2, Theorem \ref{thm: log-concavity one-point} implies that its principle eigenfunction is log-concave with respect to $g_{\mathbb M^2_K}$. Therefore, Corollary \ref{Gap-estimate} gives that \begin{equation} \label{Horoconvex gap in terms of spherical diameter}
    \Gamma(\Omega) \geq \frac{1}{\|\rho\|_\infty}\left(\frac{\pi^2}{ (D_{\mathbb M^2_K})^2}+\frac{1}{2R^2}\right)
\end{equation}
where $D_{\mathbb M^2_K}$ is the diameter of the domain in spherical geometry and $1/R^2$ is the sectional curvature of the sphere. Note that, $D_{\mathbb M^2_K}\leq D_{\mathbb H^2}$ since the conformal factor is greater or equal than $1$ and we can substitute back our values for $R$ and bounds on the weighting function to get the gap estimate stated in the introduction.

\end{proof}

\subsection{Remarks on Theorem \ref{Fundamental gap of Horoconvex-domains}}

\label{Discussion-on-horoconvexity}

Let us make several comments on the previous theorem and its proof.

\begin{remark}
It is possible to relax the hypotheses of the domains to satisfy weaker convexity properties rather than horoconvexity. Indeed, all that is needed is for the domain is convex in terms of the spherical connection.
\end{remark}

In particular, given any $\alpha>0$, we can find a diameter $D(\alpha)$ so that any domain with diameter at most $D(\alpha)$ and whose second fundamental form (with respect to the hyperbolic metric) is larger or equal than $\alpha >0$ will be convex with respect to $g_{\mathbb M^2_K}$ (see Step 2 of proof of Theorem \ref{Fundamental gap of Horoconvex-domains}). For such domains, the preceding argument shows that the fundamental gap is greater than almost $\tfrac{\pi^2}{C D^2}$, for some constant $C>1,$ approaching $1$ when the diameter approaches $0.$ In other words, we can consider the family of domains 
\[ \mathcal{C}_{D,\alpha} = \{\Omega \subset \mathbb H^2: \textup{diam}(\Omega) \leq D, \, \kappa_{\partial \Omega}\geq \alpha\}\]
 and the quantity
\[\Gamma(\mathcal{C}_{D,\alpha})=\inf_{\Omega \in \mathcal{C}_{D,\alpha}} \Gamma(\Omega) D(\Omega)^2.  \]
Our arguments show that
\[\lim_{\alpha \to 0^+} \lim_{D \to 0^+} \Gamma(\mathcal{C}_{D,\alpha}) \geq \pi^2. \]
On the other hand, \cite{bourni2022vanishing} implies that
\[ \lim_{\alpha \to  0^+} \inf_{\Omega \in \mathcal{C}_{D,\alpha}} \Gamma(\Omega) =0 \]
for all $D>0,$ and so  
\[ \lim_{D \to 0^+} \lim_{\alpha \to 0^+}  \Gamma(\mathcal{C}_{D,\alpha}) =0. \]

This shows an interesting dichotomy in that similar types of domains can have very different sizes of fundamental gap.

\subsubsection{Finding conformal factors}

 At first, the proofs of Theorem \ref{Fundamental gap of Horoconvex-domains} and Theorem \ref{Corollary to the torsion problem} might seem somewhat ad hoc. Therefore, let us conclude this paper with some intuition for how to find these results and suggestions for how to adapt them to other settings. The general philosophy throughout is to find a connection where the analysis simplifies as much as possible.

For the eigenfunction problem, connections with constant positive curvature are well-suited, since the curvature terms in \eqref{Barrier operator definition equation} simplify greatly. To take advantage of this simplification, we used stereographic projection in the proof of Theorem \ref{Fundamental gap of Horoconvex-domains} to relate the problem to spherical geometry. In this proof, we also needed a conformal model of hyperbolic space. The use of the disk model was done to further simplify the analysis, since the weighting function will be radial. However, a similar argument can be done with the half-plane model (or any other conformal model for hyperbolic geometry). This choice essentially equivalent to choosing the conformal factor and by choosing the domain in a more refined way, it might be possible to consider larger horoconvex domains.

\bibliography{references}
\bibliographystyle{alpha}
\end{document}